\documentclass{amsart}
\usepackage{amsmath, amssymb, amsthm} 
\usepackage{hyperref}
\usepackage{mathrsfs}

\newtheorem{theorem}{Theorem}
\newtheorem{lemma}[theorem]{Lemma}
\newtheorem{corollary}[theorem]{Corollary}
\newtheorem{remark}[theorem]{Remark}

\title{Some series equivalent to the extended Riemann hypothesis for Dedekind zeta functions}
\author{Vincent Nguyen}
\date{}

\keywords{extended Riemann hypothesis; Riemann xi function; Dedekind zeta function.}

\subjclass{11M26, 11R42}

\begin{document}

\begin{abstract}
    The extended Riemann hypothesis (ERH) for Dedekind zeta functions remains one of the most elusive open problems in number theory. Over the last century, many equivalent statements to the classical Riemann hypothesis alone have been discovered. We prove that the closed form of some infinite series over the non-trivial zeros of Dedekind zeta functions holds if and only if ERH is true.
\end{abstract}

\maketitle

\section{Introduction}

Let $K$ be a number field with degree $n = r_1 + 2r_2$ where $r_1$ and $r_2$ are the number of real embeddings and pairs of complex embeddings respectively. The Dedekind zeta function for a number field $K$ is defined as
$$\zeta_K(s) = \sum_{\mathfrak{a} \subseteq \mathcal{O}_K} \frac{1}{N(\mathfrak{a})^s}, \quad (\Re(s) > 1).$$
Here, the sum runs over all non-zero integral ideals $\mathfrak{a}$ of the ring of integers of $K$, denoted $\mathcal{O}_K$, and $N(\cdot)$ denotes the norm of an ideal. When $K = \mathbb{Q}$, we immediately recover the standard series representation of the Riemann zeta function, $\zeta_\mathbb{Q}(s) := \zeta(s)$. Like $\zeta$, the Dedekind zeta functions also admit a meromorphic extension to the complex plane with a simple pole at $s = 1$. The Dedekind zeta function also satisfies the functional equation \cite[p. 467]{Neukirch}
\begin{equation}\label{DedekindFunctionalEquation}
    \zeta_K(1-s) = |\Delta_K|^{s -\frac{1}{2}}\left(\cos\frac{\pi s}{2}\right)^{r_1 + r_2} \left(\sin \frac{\pi s}{2}\right)^{r_2} 2^{n(1-s)} \pi^{-ns}\Gamma(s)^n \zeta_K(s)
\end{equation}
where $\Gamma$ denotes the classical gamma function and $\Delta_K$ denotes the discriminant of the number field $K$. The functional equation \eqref{DedekindFunctionalEquation} implies Riemann's functional equation for $\zeta$ in \cite{Riemann1859}.

As with the Riemann hypothesis, there is also the extended Riemann hypothesis (ERH) for Dedekind zeta functions, which conjectures that all non-trivial zeros of the Dedekind zeta function, those zeros which lie in the critical strip $0 < \Re(s) < 1$, lie on the critical line $\Re(s) = 1/2$. Many equivalent statements of the Riemann hypothesis and its generalizations have been discovered in the past century alone. The Riemann hypothesis is equivalent to
\begin{equation}\label{RHIntegral}
    \frac{1}{\pi}\int_0^\infty \log\left|\frac{\zeta\left(\frac{1}{2} + it\right)}{\zeta\left(\frac{1}{2}\right)}\right|\frac{dt}{t^2} = \frac{\pi}{8} + \frac{\gamma}{4} + \frac{1}{4}\log(8\pi) - 2
\end{equation}
where $\gamma = 0.577\ldots$ is the Euler-Mascheroni constant. The result can be found in \cite{Hu}. A generalization of \eqref{RHIntegral} for ERH can be found in \cite{Hu2012544}. Combining \cite[Proposition 1.5.1]{Droll} with \cite[Theorem 6.5]{Dixit} implies the equivalence between the truth of ERH for $\zeta_K$ and the following equality holding
$$\sum_{\rho_K} \frac{1}{|\rho_K|^2} = \gamma_K + \frac{1}{2}\log|\Delta_K| - \frac{\gamma n}{2} - (r_1 + r_2)\log 2 - \frac{n}{2}\log \pi + 1$$
where $\gamma_K$ is the Euler-Kronecker constant of $K$, which is a generalization of $\gamma$ to arbitrary number fields. Here, the sum is understood to run over all the non-trivial zeros $\rho_K$, counting multiplicity. The equivalence for the Riemann hypothesis is noted in \cite{GunMurtyRath}. Several other equivalences for the Riemann hypothesis and its generalizations can be found in \cite{Broughan1, Broughan2}.

In 2022, Suman and Das proved in \cite{Suman} that the Riemann hypothesis is equivalent to
$$\sum_\rho \frac{1}{\left|\frac{1}{2} - \rho\right|^2} = \frac{\xi''\left(\frac{1}{2}\right)}{\xi\left(\frac{1}{2}\right)}$$
where the sum runs over the non-trivial zeros, counting multiplicity, of the Riemann zeta function and
\begin{equation}\label{RiemannXi}
    \xi(s) = \frac{1}{2}s(s-1)\pi^{-s/2}\Gamma\left(\frac{s}{2}\right)\zeta(s)
\end{equation}
is the Riemann xi function. In this spirit, we generalize Suman and Das' result to the following theorem.
\begin{theorem}\label{MainResult}
    Let $K$ be a number field. Then ERH for $\zeta_K$ is equivalent to
    \begin{equation} \label{MainResultEq}
        \sum_{\underset{\rho_K \neq \frac{1}{2}}{\rho_K}} \frac{1}{\left|\frac{1}{2} - \rho_K\right|^2} = \frac{\mathscr{X}''_K\left(\frac{1}{2}\right)}{\mathscr{X}_K\left(\frac{1}{2}\right)}.
    \end{equation}
\end{theorem}
Here, $$\mathscr{X}_K(s) := \left(s - \frac{1}{2}\right)^{-m} \xi_K(s)$$
where $m$ is the multiplicity of the zero of $\zeta_K$ at $s = 1/2$ and 
\begin{equation}\label{DedekindXi}
    \xi_K(s)\ = \frac{1}{2}s(s-1)\,|\Delta_K|^{s/2} \pi^{-n s/2} 2^{(1-s) r_{2}}\Gamma\left(\frac{s}{2}\right)^{r_{1}}\Gamma(s)^{r_{2}}\zeta_K(s).
\end{equation}
Equation \eqref{DedekindXi} is a generalization of \eqref{RiemannXi} to arbitrary number fields (see \cite{Neukirch}).

\section{Preliminaries}
Before we prove Theorem \ref{MainResult}, we must examine the behavior of $\mathscr{X}_K$. One can see that the function is entire on the complex plane. One can see from \eqref{DedekindFunctionalEquation} that $\zeta_K$ has a zero of order $r_1 + r_2 - 1$ at $s = 0$, zeros of order $r_1 + r_2$ at the negative even integers, and zeros of order $r_2$ at the negative odd integers. There are poles from $\Gamma(s/2)^{r_1}$ and $\Gamma(s)^{r_2}$, since the gamma function has simple poles at the non-positive integers. The pole at $s = 0$ is canceled out by the zero at $s=0$ from $\zeta_K$ and $s$. The simple pole at $\zeta_K$ at $s=1$ is canceled out by $s-1$. We also see that the zeros of $\mathscr{X}_K$ are precisely the non-trivial zeros of $\zeta_K$ not equal to $\frac{1}{2}$. 

Note that $\mathscr{X}_K$ does not vanish at $s = 1/2$ because $\left(s - 1/2\right)^{-m}$ cancels out the zero of $\zeta_K$ at $s=1/2$. Armitage in \cite{Armitage} constructs such number fields $K$ where $\zeta_K(1/2) = 0$. Other discussion about the vanishing of $\zeta_K$ at $s = 1/2$ can be found in \cite{Kandhil}.

Towards proving Theorem \ref{MainResult}, we prove the following lemma.
\begin{lemma}\label{OmegaSymmetry}
    $\mathscr{X}_K(s) = \mathscr{X}_K(1-s)$ for every $s \in \mathbb{C}$. 
\end{lemma}
\begin{proof}
    It is known that $\xi_K(s) = \xi_K(1-s)$. So, it suffices to show that $m$ is even. If $\zeta_K$ does not vanish at $s = 1/2$, then $m=0$. It is clear from \eqref{DedekindXi} that the multiplicity of the zero of $\xi_K$ at $s = 1/2$ is inherited from $\zeta_K$. So the order of the zero of $\xi_K(s)$ at $s= 1/2$ is $m$. Since $\xi_K(1/2 - s)$ is even and entire, there exists a sequence of complex numbers $\{a_i\}_{i=0}^\infty$ with
    $$\xi_K(s) = \sum_{k \geq 1} a_{2k} \left(s - \frac{1}{2}\right)^{2k}.$$
    Note that $a_0 = 0$ since $\xi_K$ has a zero at $s = 1/2$. The multiplicity of the zero at $s = 1/2$ is the first index $2j$ such that $a_{2j} \neq 0$. Thus, $m = 2j$ for some positive $j$. Since $m$ is even, we have that
    \begin{align*}
        \mathscr{X}_K(1-s) 
        &= \left((1-s) - \frac{1}{2}\right)^{-m} \xi_K(1-s)\\
        &= \left(\frac{1}{2} - s \right)^{-m}\xi_K(1-s)\\
        &= \left(s - \frac{1}{2}\right)^{-m} \xi_K(s)\\
        &= \mathscr{X}_K(s).
    \end{align*}
\end{proof}
\begin{lemma} \label{RealXiLemma}
    Let $\mathscr{X}_K^{(n)}(s)$ denote the $n$th derivative of $\mathscr{X}_K$. For all real $s$, $\mathscr{X}_K^{(n)}(s)$ is real.
\end{lemma}
\begin{proof}
    We first examine $\zeta_K$. Let $F(s) := (s-1)(s-1/2)^{-m}\zeta_K(s)$. We see that $F$ is entire. From the series representation of $\zeta_K$, it is easy to see that $\overline{F(s)} = F(\overline{s})$ on the half-plane $\Re(s) > 1$. Due to the Cauchy-Riemann equations, $F$ being holomorphic implies $\overline{F(\overline{s})}$ is holomorphic. Since $F(s)$ and $\overline{F(\overline{s})}$ are holomorphic and $F(s) = \overline{F(\overline{s})}$ on the open half-plane $\Re(s) > 1$, the identity theorem implies that $F(s) = \overline{F(\overline{s})}$ on the entire complex plane. One can see from the limit definition of the derivative and induction that $F^{(n)}(s) = \overline{F^{(n)}(\overline{s})}$ for each non-negative integer $n$. If we restrict $s$ to lie on the real number line, we get that $F^{(n)}(s) = \overline{F^{(n)}(s)}$. Thus, $F^{(n)}(s) \in \mathbb{R}$ for all $s \in \mathbb{R}$. So $\mathscr{X}^{(n)}_K(s) \in \mathbb{R}$ for all $s\in \mathbb{R} \setminus\{0, -1, -2,\ldots\}$ avoiding the poles from the gamma function. Because $\mathscr{X}_K$ is entire, $\mathscr{X}_K \in C^\infty(\mathbb{C})$. We immediately deduce $\mathscr{X}^{(n)}_K(s) \in \mathbb{R}$ for all $s \in \mathbb{R}$ due to continuity.
\end{proof}

The function $\xi_K$ is an entire function of order $1$ (see \cite{Barner}) and does not vanish at $s = 0$. This implies that $\mathscr{X}_K$ is also an entire function of order $1$ and does not vanish at $s=0$. Thus, we can apply Hadamard's factorization theorem to obtain
\begin{equation} \label{XiProduct}
    \mathscr{X}_K(s) = \mathscr{X}_K(0) e^{Bs}\prod_{\underset{\rho_K \neq \frac{1}{2}}{\rho_K}}\left(1 - \frac{s}{\rho_K}\right)e^{s/\rho_K}
\end{equation}
where $B$ is a constant and the product is taken over all the roots $\rho_K$ of $\xi_K$, the non-trivial zeros of $\zeta_K$ not equal to $1/2$, counting multiplicity. This product converges absolutely and uniformly on compact subsets of $U := \mathbb{C}\setminus\mathcal{Z}_K$ where $\mathcal{Z}_K$ denotes the set of non-trivial zeros of $\zeta_K$ not equal to $1/2$ (see \cite[p. 76]{Davenport2000}). The proof of the Hadamard product for the classical Riemann xi function can be found in \cite{Titchmarsh1986} and \cite{Edwards}.

\section{Main Result}
We are ready to prove Theorem \ref{MainResult}.
\begin{proof}
    From \eqref{XiProduct}, we can logarithmically differentiate to obtain
    $$\frac{\mathscr{X}_K'(s)}{\mathscr{X}_K(s)} = B + \sum_{\underset{\rho_K \neq \frac{1}{2}}{\rho_K}} \left(\frac{1}{s - \rho_K} + \frac{1}{\rho_K} \right).$$
    This series is guaranteed to converge uniformly on compact subsets of $U$. Since $\zeta_K$ is meromorphic, $\mathcal{Z}_K$ is closed and therefore $U$ is open. Thus, we can differentiate termwise again to show that
    $$\frac{\mathscr{X}_K''(s)\mathscr{X}_K(s) - \mathscr{X}_K'(s)^2}{\mathscr{X}_K(s)^2} = -\sum_{\underset{\rho_K \neq \frac{1}{2}}{\rho_K}} \frac{1}{(s - \rho_K)^2}.$$
    This series also converges uniformly on compact subsets of the open set $U$. By Lemma \ref{OmegaSymmetry}, $\mathscr{X}_K(s) = \mathscr{X}_K(1-s)$, which implies $\mathscr{X}_K'(1/2) = 0$. Thus,
    $$\frac{\mathscr{X}_K''\left(\frac{1}{2}\right)}{\mathscr{X}_K\left(\frac{1}{2}\right)} = -\sum_{\underset{\rho_K \neq \frac{1}{2}}{\rho_K}}\frac{1}{\left(\frac{1}{2} - \rho_K\right)^2}.$$
     We take the real part of both sides. By Lemma \ref{RealXiLemma}, $\mathscr{X}_K''(1/2)/\mathscr{X}_K(1/2)$ is real. Therefore,
    \begin{align*}
        \frac{\mathscr{X}_K''\left(\frac{1}{2}\right)}{\mathscr{X}_K\left(\frac{1}{2}\right)} 
        &= \Re\left(-\sum_{\underset{\rho_K \neq \frac{1}{2}}{\rho_K}}\frac{1}{\left(\frac{1}{2} - \rho_K\right)^2}\right)\\
        &= -\frac{1}{2}\sum_{\underset{\rho_K \neq \frac{1}{2}}{\rho_K}} \left( \frac{1}{\left(\frac{1}{2} - \rho_K\right)^2} + \frac{1}{\left(\frac{1}{2} - \overline{\rho_K}\right)^2}\right)\\
        &= -\frac{1}{2}\sum_{\underset{\rho_K \neq \frac{1}{2}}{\rho_K}} \frac{(1 - 2\Re(\rho_K))^2 - 2\left|\frac{1}{2} - \rho_K\right|^2}{\left|\rho_K - \frac{1}{2} \right|^4}.
    \end{align*}
    Assume ERH holds for $\zeta_K$. Then for all $\rho_K$, $\Re\left(\rho_K\right) = 1/2$. Thus,
    $$\frac{\mathscr{X}_K''\left(\frac{1}{2}\right)}{\mathscr{X}_K\left(\frac{1}{2}\right)} = -\frac{1}{2}\sum_{\underset{\rho_K \neq \frac{1}{2}}{\rho_K}} \frac{-2\left|\frac{1}{2} - \rho_K\right|^2}{\left|\frac{1}{2} - \rho_K\right|^4} 
    = \sum_{\underset{\rho_K \neq \frac{1}{2}}{\rho_K}} \frac{1}{\left|\frac{1}{2} - \rho_K\right|^2},$$
    as desired.

    Conversely, assume \eqref{MainResultEq} holds. Then
    $$\sum_{\underset{\rho_K \neq \frac{1}{2}}{\rho_K}} \frac{(1 - 2\Re(\rho_K))^2}{\left|\frac{1}{2} - \rho_K\right|^4} = 0.$$
    Since the summand is non-negative, we must have $$\frac{(1 - 2\Re(\rho_K))^2}{\left|\frac{1}{2} - \rho_K\right|^4} = 0$$
    for all non-trivial zeros $\rho_K$. This occurs if and only if $\Re(\rho_K) = 1/2$ for every non-trivial zero $\rho_K$. Thus, assuming \eqref{MainResultEq} holds implies ERH for $\zeta_K$.
\end{proof}
\begin{corollary}
    Suppose $K$ is a number field such that $\zeta_K$ does not vanish at $s = 1/2$. Then ERH for $\zeta_K$ is equivalent to
    $$\sum_{\rho_K} \frac{1}{\left|\frac{1}{2} - \rho_K\right|^2} = \frac{\xi_K''\left(\frac{1}{2}\right)}{\xi_K\left(\frac{1}{2}\right)} .$$
\end{corollary}

\begin{remark}
    If $K = \mathbb{Q}$, we immediately get Suman and Das' result in \cite{Suman}.
\end{remark}

\bibliographystyle{plain}
\bibliography{bibliography.bib}

\end{document}